\documentclass[12pt,a4paper]{article}
\usepackage{amsmath,amscd,amsthm,amssymb,stmaryrd}
\usepackage{graphicx}
\usepackage{caption}
\usepackage{subcaption}

\newtheorem{theorem}{Theorem}[section]
\newtheorem{corollary}[theorem]{Corollary}

\newtheorem{lemma}[theorem]{Lemma}
\newtheorem{proposition}[theorem]{Proposition}
\newtheorem{remark}[theorem]{Remark}

\def\r{\mathbb{R}}
\def\h{\mathbb{H}}
 \def\n{\mathbf{n}}
  \def\o{\mathbf{O}}

\title{Gradient estimates for the   constant mean curvature equation in hyperbolic space}
\author{\textbf{Rafael L\'opez}
\\
Departamento de Geometr\'{\i}a y Topolog\'{\i}a\\ Instituto de Matem\'aticas (IEMath-GR)\\
 Universidad de Granada\\
 18071 Granada, Spain}
 \date{}
\begin{document}

 
 \maketitle
 
   \begin{abstract} 
   We establish gradient estimates for solutions  to the Dirichlet problem for  the constant mean curvature  equation in hyperbolic space. We obtain these estimates on bounded strictly convex domains  by using  the maximum principles theory of $\Phi$-functions of Payne and Philippin. These   estimates are then employed to solve the Dirichlet problem when the mean curvature $H$ satisfies $H<1$ under suitable boundary conditions.    
 \end{abstract}

AMS 2010 Classification Number: 35J62, 35J25, 35J93, 35B38, 53A10

\section{Introduction and statement of the results}

  In this paper we consider the Dirichlet problem for the constant mean curvature equation on  a domain of a horosphere in three-dimensional hyperbolic space $\h^3$. In order to fix the terminology,   we consider the upper  halfspace model of   ${\mathbb H}^3$, that is, $\r^3_{+}=\{(x_1,x_2,x_3)\in\r^3:x_3>0\}$ endowed with the hyperbolic metric $g=g_0/x_3^2$,  $g_0$ being the Euclidean metric. After a rigid motion of $\h^3$, a horosphere can be expressed as a horizontal plane $P_a$ of equation $x_3=a$, $a>0$.   Let $\Omega$ be a domain of $P_a$, where we identify $\Omega$ with its orthogonal projection $\Omega\times\{0\}$ on the plane $x_3=0$. We study  the Dirichlet problem
  \begin{eqnarray}
&&\mbox{div}\left(\dfrac{Du}{\sqrt{1+|Du|^2}}\right)=-\frac{2}{u}\left(\frac{1}{\sqrt{1+|Du|^2}}-H\right)\  \mbox{in $\Omega$}\label{eq1}\\
&&u=a>0\ \mbox{on $\partial\Omega,$}\label{eq2}
\end{eqnarray}
  where $u>0$ is a smooth function in $\Omega$, $H\in\r$ is a constant and $D$ and $\mbox{div}$ denote the gradient and the divergence operators in the Euclidean plane $\r^2$.  The graph  $\Sigma_u=\{(x,u(x)):x\in\Omega\}$,  $x=(x_1,x_2)$, represents a surface in $\h^3$ with constant mean curvature $H$  computed with respect to the upwards orientation. The study of the solutions of the Dirichlet problem (\ref{eq1})-(\ref{eq2}) depends strongly of the relation between $H$ and the value $1$, the modulus of the sectional curvature $-1$ of  ${\mathbb H}^3$. For example,  if $H<1$ ($H>1$), then $\Sigma_u$ lies above the horosphere $P_a$ (respectively below $P_a$) and the geometric behaviour of $\Sigma_u$ in both cases is completely different: let us observe that in hyperbolic geometry, the translations along the   $x_3$-coordinate are not isometries of ${\mathbb H}^3$.

In this article,  we will use the theory of maximum principles developed by Payne and Philippin to obtain estimates of the gradient for a solution of (\ref{eq1})-(\ref{eq2}). We derive estimates of the gradient $|Du|$ in terms of $C^0$ bounds   of $u$.

\begin{theorem}\label{t-du}
 Let  $\Omega\subset\r^2$ be a bounded strictly convex domain. Let $u$ be a solution of   (\ref{eq1})-(\ref{eq2}) and denote $u_M=\sup_\Omega u$ and 
 $$C=\frac{1}{1-H}\frac{u_M^2}{a^2}\cdot$$
 If $0\leq H<1$ or  if $H<0$ with
 \begin{equation}\label{um}
 u_M<\sqrt{\frac{H-1}{H}}a,
 \end{equation}
 then   
 \begin{equation}\label{duu}
 |Du| \leq \frac{\sqrt{C^2-(1+HC)^2}}{1+HC}\quad\mbox{in $\Omega$.}
\end{equation}

\end{theorem}

If we have estimates for the gradient of solutions of (\ref{eq1})-(\ref{eq2}),   it is natural to address  the problem of  the  existence of solutions of the Dirichlet problem. In the context of the hyperbolic space, the results of existence require   some assumption on the convexity of the domain $\Omega$.  If $0\leq H<1$, the convexity of $\partial\Omega$ is enough to ensure the existence of a solution of (\ref{eq1})-(\ref{eq2}): see \cite{lo0,ns}. However, if $H<0$,  the mere convexity of $\Omega$ does not ensure the existence of solutions and it is required stronger convexity. More exactly,   the solvability of the Dirichlet problem    (\ref{eq1})-(\ref{eq2}) was proved if the curvature $\kappa$ of $\partial\Omega$ satisfies $-k<H<0$  (\cite[Th. 1.1]{lm}). For other existence results, see \cite{dl,li,lo1,st}. As a consequence of Theorem \ref{t-du}, we establish the following existence result.

  \begin{theorem}\label{t-ex}Let $\Omega\subset\r^2$ be a bounded strictly convex domain.  Let $2R$ be the  diameter of   $\partial\Omega$. If $-1\leq H<0$ satisfies     
\begin{equation}\label{hhh}
R^2<-2-\frac{1}{H}+2\sqrt{\frac{H}{H-1}},
\end{equation}
then there exists a unique solution of (\ref{eq1})-(\ref{eq2}).
  \end{theorem}
We notice that we need to assume that  the diameter of $\Omega$ is   small in relation with the value of $H$ but, in contrast, it is not necessary strong convexity of $\partial\Omega$ and  we allow that the existence of regions of $\partial\Omega$ whose curvature   is closed to $0$. 

Theorems \ref{t-du} and \ref{t-ex} will be proved in $\S$  \ref{sec3}. In the proof of these results, we need to show the uniqueness of critical points of solutions of  (\ref{eq1})-(\ref{eq2}). Although this may be expected because the resemblance of (\ref{eq1}) with other quasilinear elliptic equations, as for example, the capillary equation (\cite{ba,be,el}) or the singular minimal surface equation (\cite{lo2}), we prove this uniqueness only in the range of values $H<1$, which is enough for our purposes: see Theorem \ref{t1} in $\S$ \ref{sec2}.  Finally we prove in $\S$ \ref{sec4}   an estimate from below of the global maximum $u_M$ of a solution of (\ref{eq1})-(\ref{eq2}) when $H<1$.  In general, estimates of $u$ are obtained by comparing $u$ with known solutions of (\ref{eq1}), as for example, radial solutions. However, our result  establishes a lower estimate of the value $u$ at the critical point in terms of the curvature of $\partial\Omega$ and $H$: see Theorem \ref{t3}.

 \section{Uniqueness of critical points}\label{sec2}
 
  The first result in this paper establishes, under some hypothesis, the uniqueness of critical points of a solution of the Dirichlet problem.  The topic  on  the number of critical points  of solutions for elliptic equations is a subject of high interest and the literature is very extensive, especially related with  the question of the convexity of level sets of solutions of elliptic equations.   In the context of the constant mean curvature equation in Euclidean space, and if the domain is convex,  Sakaguchi proved the existence of a unique critical point assuming Dirichlet boundary condition or Neumann boundary condition (\cite{sa}). In this   paper we address this problem for the constant mean curvature equation  in hyperbolic space when $H<1$.  
  
 \begin{theorem}\label{t1}
 Let $\Omega\subset\r^2$ be a bounded strictly convex domain and let $H\in\r$.  If $H<1$, then  a solution $u$ of (\ref{eq1})-(\ref{eq2}) has exactly one critical point, which coincides with the point where $u$ attains its global maximum.
\end{theorem}

We prove this result as a consequence of the following arguments. 

A first step consists in proving the existence of at least one critical point of a solution $u$ of (\ref{eq1})-(\ref{eq2}). When $H\leq 0$, this is achieved by the Hopf maximum principle. Indeed, because the right-hand side of (\ref{eq1}) is non positive, the minimum of $u$ is attained at some boundary point, proving $u>a$ in $\Omega$. Since $\Omega$ is bounded, the function $u$ has a global maximum at some interior point. 

This argument fails if $0<H<1$. For this range of values of $H$ (also if $H\leq 0$) we will use a {\it comparison principle}  based in the standard theory of quasilinear elliptic equations (\cite[Th. 10.1]{gt}). In our context, it can be formulated as follows:   if  two surfaces $\Sigma_1$ and $\Sigma_2$ have a common interior point $p$ and with constant mean curvature $H_1$ and $H_2$, respectively,  with respect to the orientations that coincide at $p$,   if $\Sigma_1$ lies above $\Sigma_2$ around $p$, then $H_1\geq H_2$ (the same conclusion holds if $p$ is a common boundary point with tangent boundaries at $p$): see \cite[p. 194]{lo}.

 \begin{lemma}\label{l1}
  Suppose $\Omega\subset\r^2$ is a bounded domain. If $H<1$, then a solution of (\ref{eq1})-(\ref{eq2}) satisfies $u>a$ in $\Omega$.
 \end{lemma}
 \begin{proof}   The proof is by contradiction. Suppose that there exists   $x_0\in \Omega$ such that $u$ attains the minimum value, $u(x_0)\leq a$.    Let $p=(x_0,u(x_0))$. For $b<u(x_0)$, consider the  the horosphere $P_b$ of equation $x_3=b$,  whose mean curvature is $1$ with respect to the upwards orientation. Then we move up $P_b$ by letting $b\nearrow\infty$, until the first touching point with $\Sigma_u$ at $b_1=u(x_0)$. Then the horosphere $P_{b_1}$ touches $\Sigma_u$ at   $p$, which is an interior point of both   $\Sigma_u$ and $P_{b_1}$. As $\Sigma_u$  lies above $P_{b_1}$, we arrive a contradiction with the comparison principle.  
 \end{proof}
  
Once proved this lemma, we follow with the proof of Theorem \ref{t1}.  Denote $u_k=\partial u/\partial x_k$, $k=1,2$,  and consider the summation convention of repeated indices. Equation (\ref{eq1}) can be expressed as
 $$(1+|Du|^2)\Delta u-u_iu_ju_{ij}+\frac{2(1+|Du|^2)}{u}-\frac{2H(1+|Du|^2)^{3/2}}{u}=0.$$
 Denote $v^k=u_k$, $k=1,2$, and we differentiate the above identity with respect to $x_k$ obtaining:
  \begin{eqnarray}
 &&\left((1+|Du|^2)\delta_{ij}-u_iu_j\right)v_{ij}^k+2\left(u_i\Delta u-u_ju_{ij}+\frac{2 u_i}{u}-\frac{3Hu_i}{u}(1+|Du|^2)^{1/2}\right)v_i^k\nonumber\\
 &&-\frac{2(1+|Du|^2)}{u^2}(1-H\sqrt{1+|Du|^2})v^k=0\label{eq33}
 \end{eqnarray}
 for $k=1,2$ and where $\delta_{ij}$ is the Kronecker delta. Equation (\ref{eq33}) is a linear elliptic equation in $v^k$. We   need to  apply the Hopf Maximum Principle  (\cite[Th. 3.5]{gt}) to this equation. Then we have to know   that the term of $v^k$ is non positive, or equivalently,  
  \begin{equation}\label{ineh}
  1-H\sqrt{1+|Du|^2} \geq 0\quad \mbox{in $\Omega$.}
  \end{equation}
  If $H\leq 0$, this inequality is clear. If $0<H< 1$,  one needs to   estimate   $|Du|$ in terms of $H$. For this, we prove the next lemma, which is implicitly contained in the proof of the main result in \cite{lm}.

 \begin{lemma}\label{l2}
Let $\Omega\subset\r^2$ be a bounded strictly convex domain of $\r^2$ and let $0<H< 1$. If $u$ satisfies (\ref{eq1})-(\ref{eq2}), then
\begin{equation}\label{esh}
|Du|^2\leq\frac{1-H^2}{H^2}\cdot
\end{equation}
  \end{lemma}
  
 \begin{proof} 
 Consider the Minkowski model for $\h^3$ (see notation and details in \cite{lm})).   It is proved in \cite[Theorem  4.1]{lm} that under the assumptions of Lemma \ref{l2}, 
 \begin{equation}\label{hna}
H\langle p,a\rangle+\langle N'(p),a\rangle\leq 0,\quad \mbox{ $p\in \Sigma_u$},
\end{equation}
 where $N'$ is the Gauss map of $\Sigma_u$.  We   write the inequality (\ref{hna}) in the upper half-space model of $\h^3$. The relation between both models establishes
 $$ \langle p,a\rangle=\frac{1}{u},\quad  \langle N',a\rangle=-\frac{\langle N,(0,0,1)\rangle}{u},$$ where here $N$ is the Gauss map of $\Sigma_u$ as surface in Euclidean space $\r^3_+$. Thus (\ref{hna}) becomes 
 $H-\langle N,(0,0,1)\rangle \leq 0$, that is, 
 $$H-\frac{1}{\sqrt{1+|Du|^2}}\leq 0,$$
 which yields (\ref{esh}).
 \end{proof}

As a consequence of Lemma \ref{l2},   the   Hopf Maximum Principle  for equation  (\ref{eq33}) implies that  if $v^k$ takes a non-negative maximum in $\Omega$ or a non   positive minimum in $\Omega$, then $v^k$ must be a constant function (\cite[Th. 3.5]{gt}). We point out also that   the function $u$ is    analytic by standard theory (\cite{bers,ni}), and the same holds for the functions $v^k$.

 For each  $\theta\in\r$, let $(\cos\theta,\sin\theta)$ be a vector of the unit circle ${\mathbb S}^1$. Since (\ref{eq33}) is a linear equation on $v^k$, the function  
 \begin{equation}\label{vv}
 v(\theta)=  Du\cdot (\cos\theta,\sin\theta)=v^1\cos\theta+v^2\sin\theta
 \end{equation}
 also satisfies (\ref{eq33}).  Denote $\n$ the outward unit normal vector of $\partial\Omega$. Because $u$ is constant along $\partial\Omega$, we have $(v^1,v^2)=Du= ( Du\cdot\n ) \n$ along $\partial\Omega$, that is,
 $$(v^1,v^2)=\frac{\partial u}{\partial n}\n.$$
From (\ref{vv}), 
 $$v(\theta)=\frac{\partial u}{\partial\n} \n\cdot (\cos\theta,\sin\theta)\quad \mbox{ along }\partial\Omega.$$
 On the other hand, since $u$ is constant along $\partial \Omega$, the Strong Maximum Principle of  Hopf  (\cite[Th. 3.6]{gt})  implies that any outward
directional derivative on $\partial\Omega$ is negative and thus, 
$$\frac{\partial u}{\partial \n}<0\quad \mbox{along  $\partial\Omega$}.$$
  Fix $\theta\in\r$. Since $\partial\Omega$ is strictly convex, the map $\n:\partial\Omega\rightarrow{\mathbb S}^1$ is one-to-one. It follows that  there exist exactly two points of $\partial\Omega$ where  $\n(s)$ is orthogonal to the fixed direction $(\cos\theta,\sin\theta)$.  By the definition of $v(\theta)$, the function $v(\theta)$ vanishes along $\partial\Omega$ at exactly   two points. 
   
We now follow the argument given by Philippin in  \cite{ph} to prove the uniqueness of the critical points. By completeness, we give it briefly. The proof is by contradiction, and suppose that there are at least two critical points of $u$ in $\Omega$. Let $P_1$ and $P_2$ be two critical points which are fixed in the sequel. The argument consists into the following steps.

\begin{enumerate}
\item The function $v(\theta)$ is not constant in $\Omega$ because $v(\theta)$ only has two zeroes along $\partial\Omega$.
\item As a consequence, the critical points of $v(\theta)$ are isolated point of $\Omega$ because $v(\theta)$ is analytic. 
\item Let $\mathcal{N}_\theta=v(\theta)^{-1}(\{0\})$ be the nodal set of $v_\theta$. Because $v(\theta)$ is analytic, standard theory asserts that near to a critical point of $v(\theta)$, the function $v(\theta)$ is asymptotic to a harmonic homogeneous polynomial (\cite{bers}). Following Cheng \cite{ch}, $\mathcal{N}_\theta$  is diffeomorphic to the nodal set of the homogeneous polynomial  that approximates, in particular,     $\mathcal{N}_\theta$ is formed by a set of regular analytic curves at regular points, the so-called nodal lines. On the other hand, in a neighbour of a critical point, the   nodal lines   form an equiangular system. 

We claim that  there does not exist a closed component of $\mathcal{N}_\theta$ contained in $\Omega$. This is because if   $\mathcal{N}_\theta$ encloses a subdomain $\Omega'$ of $\Omega$ where $v(\theta)=0$ along $\partial\Omega'$,  the maximum principle   implies that $v(\theta)$ is identically $0$ in $\Omega'$, a contradiction.  

\item  We prove that $\mathcal{N}_\theta$ is formed exactly by one nodal line.  Suppose by contradiction that  there are two nodal lines  $L_1$ and $L_2$. Because both $L_1$ an $L_2$  are not closed, then the arcs $L_1$ and $L_2$ end at the boundary points where $v(\theta)$ vanishes, being both points the two end-points of $L_i$. Since $\Omega$ is simply-connected, then the two arcs $L_1$ and $L_2$ enclose at least one subdomain $\Omega'\subset\Omega$ where $v(\theta)$ vanishes  along $\partial\Omega'$. This is impossible by the maximum principle.

\item As a conclusion, the nodal set $\mathcal{N}_\theta$ is formed exactly by one arc. We now give an orientation to the arc $\mathcal{N}_\theta$ for all $\theta$.  The chosen orientation in   $\mathcal{N}_\theta$   is that we first pass through $P_1$ and then through $P_2$. With respect to this orientation, we are ordering the two boundary points where $v(\theta)$ vanishes. More precisely, denote by   $P(\theta)$ the initial point of $\mathcal{N}_\theta$, which after passing $P_1$ and then $P_2$, finishes at the other boundary point, which is denoted by $Q(\theta)$. 

\item Let us consider $\theta$ varying in the interval $[0,\pi]$.  We observe that by the definition of $v(\theta)$ in (\ref{vv}),   the functions $v(0)$ and $v(\pi)$ coincides up to the sign, that is, $v(0)=-v(\pi)$ and thus the nodal lines $\mathcal{N}_0$ and $\mathcal{N}_\pi$ coincide as sets of points. However,   when    $\theta$ runs in $[0,\pi]$, the ends points of $\mathcal{N}_0$ interchange its position when $\theta$ arrives to the value $\theta=\pi$, that is, in the nodal line $\mathcal{N}_\pi$. Consequently, and according to the chosen orientation in $\mathcal{N}_\theta$,  $P(0)=Q(\pi)$ and $P(\pi)=Q(0)$. Because  all the arcs $\mathcal{N}_\theta$ pass first $P_1$ and then $P_2$, this interchange of the end points between $\mathcal{N}_0$ and $\mathcal{N}_\pi$ implies the existence of another nodal line for  $v(\pi)$.  This is impossible by the item 4: this contradiction completes the proof of Theorem \ref{t1}.

\end{enumerate}

We extend  Theorem \ref{t1} in the limit case $u=0$ along $\partial\Omega$.

\begin{corollary}   Let $\Omega\subset\r^2$ be a bounded strictly convex domain.  Let $H$ be a real number with $H<1$. If $u$ is a solution  (\ref{eq1}) and $u=0$ along $\partial\Omega$, then $u$ has a unique critical point.
\end{corollary}

\begin{proof} 
We consider positive values $a$ sufficiently close to $0$ so the set $\Omega_a=\{x\in\Omega: u(x)> a\}$ is strictly convex. Then Theorem  \ref{t1} asserts the existence of a unique critical point, which must coincide for all $a$ because $\Omega_{a'}\subset\Omega_a$ if $a<a'$. The argument finishes by letting $a\rightarrow 0$. 
\end{proof}

\section{Proof of theorems \ref{t-du} and \ref{t-ex}}\label{sec3}

In this section we apply the theory of the maximum principle developed by  Payne and Philippin in \cite{pp} for some $\Phi$-functions associated to equation (\ref{eq1}).  We introduce the notation employed there.  Consider an equation of type
\begin{equation}\label{ep}
\mbox{div}(g(q^2)Du)+\rho(q^2)f(u)=0, 
\end{equation}
where $\rho, g>0$, $g$ is a $C^2$ function of its argument and $\rho$ and $f$ are $C^1$ functions. Here $q=|Du|$. We also assume that (\ref{ep}) satisfies the elliptic condition $g(\xi)+2\xi g'(\xi)>0$ for all $\xi>0$. We define the $\Phi$-function 
$$\Phi(x;\alpha)=\int_{c_1}^{q^2}\frac{g(\xi)+2\xi g'(\xi)}{\rho(\xi)}\, d\xi+\alpha\int_{c_2}^u f(\eta)\,d\eta,$$
where $\alpha$ is a  real parameter and $c_1,c_2\in\r$. Here   the functions $g$ and $\rho$ are evaluated in $q^2$.   

We now  prove Theorem \ref{t-du}.

\begin{proof}[Proof of Theorem \ref{t-du}]

 For equation  (\ref{eq1}), we take  $c_1=0$, $c_2=1$ and the functions $g$, $\rho$ and $f$ are
\begin{equation}\label{cho}
g(\xi)=\frac{1}{\sqrt{1+\xi}},\ \rho(\xi)=\frac{1}{\sqrt{1+\xi}}-H,\ f(u)=\frac{2}{u}\cdot
\end{equation}
Following the theory of Payne and Philippin, we require that $\rho$ is positive, which it is clear if   $H\leq 0$.  On the other hand, in the range $0<H<1$, the evaluation of $\rho$ at $q^2$ is non-negative by Lemma \ref{l2}. A straight-forward computation of $\Phi(x;\alpha)$  gives  
$$\Phi(x;\alpha)=\log\left(\frac{(1+q^2)}{(1-H\sqrt{1+q^2})^2} u^{2\alpha}\right),\quad x\in\Omega.$$

 When $\Omega$ is strictly convex, it is proved in \cite[Corollary    1]{pp} that $\Phi(x;2)$ attains its maximum at one critical point of $u$. By Theorem  \ref{t1}, we know that the function $u$ has exactly one critical point, which we denote by $\o$, and let $u_M=u(\o)$, which coincides with the maximum value of $u$ in $\Omega$. Then we find 
$$\frac{1+|Du|^2}{(1-H\sqrt{1+|Du|^2})^2}u^4 \leq \frac{1}{(1-H)^2}u_M^4,$$
that is,
\begin{equation}\label{du1}
\frac{1+|Du|^2}{(1-H\sqrt{1+|Du|^2})^2} \leq  \frac{1}{(1-H)^2}\left(\frac{u_M}{u}\right)^{4}\leq \frac{1}{(1-H)^2}\left(\frac{u_M}{a}\right)^{4}.
\end{equation}
Recall the value $C=u_M^2/((1-H)a^2)$.  It follows from (\ref{du1}) that
$$(1+HC)\sqrt{1+|Du|^2}\leq C.$$
The inequality (\ref{duu}) is shown provided $1+HC>0$. This inequality is immediate if $0\leq H<1$. In case $H<0$,  the inequality $1+HC>0$  is equivalent to (\ref{um}).
\end{proof}

 We follow by  focusing in    Theorem 4 of \cite{pp}.   The  inequality (2.39) in \cite{pp} can be written for our functions defined in (\ref{cho})  as 
\begin{equation}\label{ff1}
\left(\delta_{ij}-\frac{u_iu_j}{1+|Du|^2}\right)\Phi_{ij}+{W}_i\Phi_i\geq \frac{2 (\alpha -1) \left(2 H \sqrt{1+q^2}+(\alpha -2) q^2-2\right)}{u^2 (1+q^2)},
\end{equation}
where ${W}_i$ is a vector function uniformly bounded in $\Omega$. In order to apply the First Hopf Maximum Principle, we require that the right-hand side in (\ref{ff1}) is non-negative. If $\alpha\in [0,1]$, it suffices that the expression in the second parentheses in (\ref{ff1}) is non-positive. This is clear if $H\leq 0$ independently if $\Omega$ is or is not convex. If $0<H<1$ and $\Omega$ is convex,  we deduce  from (\ref{esh}) that
$$2 H \sqrt{1+q^2}+(\alpha -2)q^2-2\leq (\alpha-2)q^2\leq 0.$$ 
Following \cite{pp}, we deduce that $\Phi(x;\alpha)$ attains its maximum at some  some boundary point for all $\alpha\in [0,1]$. 

In the particular case $\alpha=0$, we deduce the following result.

\begin{corollary}\label{t2} Let $\Omega\subset\r^2$ be a bounded  domain and let $H\leq 0$. If $u$ is a solution of (\ref{eq1}),  
$$\max_{\overline{\Omega}}|Du|=\max_{\partial\Omega}|Du|.$$
The same holds when $0<H<1$ if, in addition,  $\Omega$ is  strictly convex, and $u=a>0$ on $\partial\Omega$.
\end{corollary}

 \begin{proof}  If we take $\alpha=0$ in (\ref{ff1}), then there exists a boundary point $P\in\partial\Omega$ such that 
 $$\frac{1+|Du|^2}{(1-H\sqrt{1+|Du|^2})^2}\leq \frac{1+q_M^2}{(1-H\sqrt{1+q_M^2})^2},$$
 where $q_M=|Du|(P)$. It follows directly that $|Du|\leq q_M$, proving the result. 
\end{proof}

From Theorem \ref{t-du} and Corollary  \ref{t2}, we prove the existence result of Theorem \ref{t-ex}.

 \begin{proof}[Proof of Theorem \ref{t-ex}]
  
  The uniqueness of solutions is a consequence that the right-hand side of (\ref{eq1}) is non-decreasing on $u$ by Lemma \ref{l2} (\cite[Th. 10.2]{gt}).
  
 For the existence of   a solution $u$ of (\ref{eq1})-(\ref{eq2}), we apply a modified version of the continuity method to the family of Dirichlet problems parametrized by $\tau\in [0,1]$  
 $$\mathcal{P}_\tau: \left\{\begin{array}{cll}
Q_\tau[u]&=&0 \mbox{ in $\Omega$}\\
 u&=&a \mbox{ on $\partial\Omega,$}
 \end{array}\right.$$
 where 
 $$Q_\tau[u]= \mbox{div}\left(\dfrac{Du}{\sqrt{1+|Du|^2}}\right)+\frac{2}{u}\left(\frac{1}{\sqrt{1+|Du|^2}}- \tau H\right).$$
 See \cite[Th. 11.4]{gt}. The graph $\Sigma_{u_\tau}$ of a solution of $u_\tau$ of $\mathcal{P}_\tau$ is a graph on $P_a$ with constant mean curvature $ \tau H$ and boundary $\partial\Omega$. Let us observe that for the value $\tau=0$, there is a solution of   $\mathcal{P}_0$ because $\partial\Omega$ is convex (\cite{lm,ns}). As usual, let $\mathcal{A}$ be the subset of $[0,1]$ consisting of all $\tau$ for which  the Dirichlet problem  $\mathcal{P}_\tau$ has a $C^{2,\alpha}$ solution. The proof consists in  showing that $1\in \mathcal{A}$ because standard regularity PDE results guarantee that any solution of $Q_\tau[u]=0$ is smooth in $\Omega$.

First observe that $\mathcal{A}\not=\emptyset$ because $0\in\mathcal{A}$. On the other hand, the set $\mathcal{A}$ is open in $[0,1]$ because 
$$\frac{\partial Q_\tau[u]}{\partial u}=-\frac{2}{u^2}\left(\frac{1}{\sqrt{1+|Du|^2}}- \tau H\right)\leq 0,$$
since $H<0$.

Finally, the main difficulty lies in proving that $\mathcal{A}$ is closed. This follows if   we are able to derive a priori $C^0$ and $C^1$ estimate of $u_\tau$ for every $\tau\in [0,1]$ and depending only on the initial data. In other words, we have to find a constant $M$, depending only on $H$, $a$ and $\Omega$, such that if $u_\tau$ is a solution of $\mathcal{P}_\tau$, then 
\begin{equation}\label{sup}
\sup_\Omega |u_\tau|+\sup_\Omega|Du_\tau|\leq M.
\end{equation}
See \cite[Th. 13.8]{gt}.  However, by using Theorem \ref{t-du}, it is enough if we find an upper bound for $\sup_\Omega |u_\tau|$. We  now use a geometric viewpoint of the solutions of $\mathcal{P}_\tau$. 

Fix $H\in\r$. After a  dilation from the origin of $\r^3_+$, which is an isometry of $\h^3$, we suppose $a=1$. Then the diameter of  $\partial\Omega$ coincides with the Euclidean one.  Let $C_R\subset P_1$ be  the circumscribed circle of $\partial\Omega$, which has a radius equal to $R$. After a horizontal translation in $\r^3_+$,  if necessary, we suppose that the centre of $C_R$ is $(0,0,1)$ and denote $D_R\subset P_1$ the disc bounded by $C_R$, which contains $\Omega$ in its interior.   We know that $\Sigma_u$ lies above the plane $x_3=1$. On the disc $D_R$, we are going to place an umbilical surface $\Sigma_w$ with the same mean curvature $H$ and being a graph on $D_R$. Indeed, and from the Euclidean viewpoint, $\Sigma_w$ is a spherical cap which is a graph of a function $w$ on the disc $D_R$. Then we prove that $\Sigma_u$ lies in the interior of the domain determined by $\Sigma_w$ and the plane $P_1$, or in other words, $u<w$ in $\Omega$. This will be proved by doing dilations $p\rightarrow tp$, $p\in\r^3_+$ from the origin $O$ of $\r^3$. After that, we have $u_M<w_M$, where $u_M$ and $w_M$ are the global maximum of $u$ and $w$ respectively. But now, we notice that $w_M$ depends only on the initial data, that is, from $\Omega$, $a$ and $H$.

The first step is to show the existence of the surface $\Sigma_w$. Consider (part of) the   Euclidean sphere in $\r^3_+$ of radius $m>0$ given by 
$$w(r)=c_0+\sqrt{m^2-r^2},\quad 0\leq r\leq R,$$
where
\begin{equation}\label{eqmh}
c_0=-mH,\quad m^2=(1-c_0)^2+R^2,
\end{equation}
 $0<c_0<1$ and $w(R)=1$. The mean curvature of $\Sigma_w$ is $H$ with respect to the upwards orientation. If we see $c_0$ as a parameter varying from $0$ to $1$, the value of the mean curvature of $\Sigma_w$  goes from $0$ to $-1/R$. It is not difficult to see that the right-hand side of (\ref{hhh}) is less than $1/H^2$. Thus $R^2<1/H^2$, that is, $-1/R<H$. Definitively, given $H$ under the hypothesis of Theorem \ref{t-ex}, we have assured the existence of $\Sigma_w$. 

We now do the argument of comparison between $\Sigma_u$ and $\Sigma_w$ by dilations. By dilations of $\Sigma_w$ with respect to the origin $O$ of $\r^3_+$, namely, $t\Sigma_w$ and $t>1$, we take $t$ sufficiently big so $t\Sigma_w$ does not meet $\Sigma_u$. Then let $t\searrow 1$ until the first touching point between $t\Sigma_w$ with $\Sigma_u$. Because an interior touching point is not possible   because both surfaces have the same (constant) mean curvature, then the first touching point occurs at $t=1$, that is, when $\Sigma_w$ comes back to its initial position and  $\Sigma_w$ touches  $\Sigma_u$ only at some boundary point of $\Omega$. In particular, $\Sigma_u$ is contained inside the domain determined by $\Sigma_w$ and the plane $x_3=1$. Therefore, we deduce that the global maximum $u_M$ is less than the highest point of $\Sigma_w$, namely, $w_M=c_0+m=m(1-H)$ and 
$$u_M<m(1-H).$$
The above argument has been done for the value $H$, but it holds for $\tau H$, $\tau\in [0,1]$. Indeed, we replace $H$ by $\tau H$. We now prove the $C^0$ estimates for the problems $\mathcal{P}_\tau$. Fix $-1\leq H<0$ and let $u_\tau$ the solution of $\mathcal{P}_\tau$, $\tau\in [0,1]$.  Let us observe that the mean curvature of $\Sigma_{u_\tau}$ is $\tau H$ and $\tau H>H$ for $\tau\in [0,1)$. Then the same process of dilations  together the comparison  principle proves   that for each $\tau \in [0,1]$, we find
\begin{equation}\label{m1}
u_\tau< w_M=m (1-H).
\end{equation}
 In order to use Theorem \ref{t-du}, and because $H<0$ and $a=1$, it suffices to prove 
\begin{equation}\label{m1h}
m(1-H)<\sqrt{\frac{H-1}{H}},
\end{equation}
that is, 
\begin{equation}\label{inedsi}
m<\frac{1}{\sqrt{H^2-H}}\cdot
\end{equation}
However, from  (\ref{eqmh}), we deduce $m^2=(1+mH)^2+R^2$, which leads to 
$$m=\frac{H+\sqrt{H^2+(1-H^2)R^2}}{1-H^2}\cdot$$
By using  (\ref{hhh}), we conclude the   desired inequality (\ref{inedsi}). Once we have obtained  (\ref{m1h}),   Theorem \ref{t-du} applies deducing   an a priori estimate for $|Du|$. Hence, and together (\ref{m1}), we have proved the existence of $M$ in (\ref{sup}). This completes the proof of Theorem \ref{t-ex}.
  
     \end{proof}
 
  \begin{remark} We compare this result with Theorem 1.1 in \cite{lm}. In \cite{lm}, the hypothesis  requires that $\Omega$ is strongly convex   in terms of the boundary data $H$, namely, $\kappa>|H|$. However in Theorem \ref{t-ex} we need that the domain $\Omega$ is strictly convex but it may contain regions where the curvature $\kappa$ of $\partial\Omega$ is close to $0$. In contrast, the size of the domain is small in relation to the value of  $H$. 
 \end{remark} 
    
\section{A lower estimate of the critical point }\label{sec4}

In this section, for $H<1$, we prove  an estimate from below of the global maximum of a solution of (\ref{eq1})-(\ref{eq2}).

\begin{theorem}\label{t3} Let $\Omega\subset\r^2$ be a  bounded strictly convex domain   with curvature $\kappa>0$. If  $H<1$ and $u$ is a solution of (\ref{eq1})-(\ref{eq2}),   then
\begin{equation}\label{um2}
u_M\geq \frac{1-H}{\kappa_0}, 
\end{equation}
where $\kappa_0=\max_{\partial\Omega}\kappa$.
\end{theorem}

 Firstly, we need to prove a minimum principle for the function $\Phi(x;1)$. The next result is inspired by other similar in the torsional creep problem (\cite{ph}).

\begin{proposition}\label{t-min}
 Let $\Omega\subset\r^2$ be a bounded strictly convex domain. Let $H$ be a real number with $H<1$. If $u$ is a non-radial solution of (\ref{eq1})-(\ref{eq2}), then the function $\Phi(x;\alpha)$ attains its minimum value on $\partial\Omega$ for any  $\alpha\in [1,2]$.
\end{proposition}

\begin{proof}  Following \cite[inequality (2.15)]{pp},    it was proved that if $u$ is a solution of (\ref{eq1})-(\ref{eq2}), then $\Phi(x;\alpha)$ satisfies the next elliptic differential equation
  \begin{equation}\label{pp2}
\left(\delta_{ij}-\frac{u_iu_j}{1+|Du|^2}\right)\Phi_{ij}+\tilde{W}_i\Phi_i=\frac{2 (\alpha -2) (\alpha -1) \left(2(1- H \sqrt{q^2+1})+q^2\right)}{\left(q^2+1\right) u^2},
\end{equation}
where $\tilde{W}_i$ is a vector function which is singular at the critical point of $u$. It is not difficult to see that if  $\alpha\in [1,2]$, the  right-hand side of (\ref{pp2}) is non-positive because $(\alpha-2)(\alpha-1)\leq 0 $  and the expression in parentheses $2(1-H\sqrt{1+q^2})+q^2$ is always non-negative: this is immediate for $H\leq 0$ and if $0<H<1$, we use Lemma \ref{l2}. 

By the Hopf Maximum Principle, and since the vector functions $\tilde{W}_i$ are singular at the critical points of $u$, we conclude that $\Phi(x;\alpha)$ attains its minimum at the unique critical point of $u$ or at a boundary point. Recall that by  Theorem  \ref{t1}, the function $u$ has exactly one critical point $\o$. 

The proof of Proposition \ref{t-min} finishes if we discard the case that the minimum occurs at some critical point. The proof follows now the next steps. 
\begin{enumerate}
\item The function $\Phi(x;\alpha)$ is not constant in $\Omega$. The proof is by contradiction. If $\Phi$ is constant, then the left-hand side of (\ref{pp2}) is $0$. If we see the right-hand side of (\ref{pp2}), the only possibility to be $0$ is that   $\alpha$ is $1$ or $2$. We prove that this is not possible. We consider the case   $\alpha=1$ because the argument for $\alpha=2$ is similar.    By the expression of $\Phi(x;1)$, we find that 
$$\frac{1+|Du|^2}{(1-H\sqrt{1+|Du|^2})^2}u^{2}$$
is constant, in particular, $|Du|$ is constant along $\partial\Omega$. Since $u=a$ along $\partial\Omega$, then $\partial u/\partial\n$ is constant along $\partial\Omega$. Then $u$ is a solution of the Dirichlet problem (\ref{eq1})-(\ref{eq2}) together the Neumann condition $\partial u/\partial\n=\mbox{ct}$ along $\partial\Omega$.  A result of Serrin establishes that $\Omega$ is a round disk and $u$ is a radial function $u=u(r)$ (\cite{se}).  This is a contradiction.
\item  After a change of coordinates, suppose that the critical point is $\o=(0,0)$. Then we deduce $u_1(\o)=u_2(\o)=0$. A new change of coordinates allows to assume $u_{12}(\o)=0$. Since $u$ is a maximum of $u$, we have $u_{11}(\o)\leq 0$ and $u_{22}(\o)\leq 0$. 

{\it Claim:} $u_{11}(\o)< 0$ and $u_{22}(\o)< 0$. 

The proof is by contradiction and suppose that $u_{11}(\o)=0$ (the same argument if $u_{22}(\o)=0$).  Here we follow the same notation as in the proof of Theorem  \ref{t1}. If the function $v^1=u_1$ is constant in $\Omega$, then $u$ depends only on the variable $x_2$ and the boundary condition (\ref{eq2}) is impossible. Thus $v^1$ is a non constant analytic function. Since $v^1$ vanishes at $\o$ as well as $v^1_1$ and $v^1_2$, the function $v^1$ vanishes at $\o$ with a finite order $m\geq 1$. Thus there exist at least two nodal lines of  $v^1$ which form an equiangular system in a neighbour of $\o$. We have proved in Theorem  \ref{t1} that this is impossible because there exists exactly one nodal   line of $v^1$.
\item Finally we prove that $\Phi(x;\alpha)$ can not attain its minimum at $\o$. We know $u_1(\o)=u_2(\o)=u_{12}(\o)=0$. We need the first and second partial derivatives of $\Phi$ at   $\o\in\Omega$. Following the notation employed in \cite[p. 197]{pp}, at the critical point $\o$ we have
$$ \Phi_i(\o,\alpha)=0,\quad \Phi_{ij}(\o;\alpha)=2\frac{g+2q^2g'}{\rho}u_{ik}u_{jk}+\alpha fu_{ij}.
$$
Hence,    and from (\ref{eq1}),  
\begin{eqnarray*}
\Phi_{11}(\o;\alpha)&=&\frac{2}{1-H}u_{11}(\o)^2+\frac{2\alpha}{u(\o)}u_{11}(\o)\\
\Phi_{12}(\o;\alpha)&=&0\\
\Phi_{22}(\o;\alpha)&=&\frac{2}{1-H}u_{22}(\o)^2+\frac{2\alpha}{u(\o)}u_{22}(\o).
\end{eqnarray*}
Because $\o$ is a minimum of   $\Phi(x;\alpha)$, we find that $\Phi_{11}(\o;\alpha)\geq 0$ and $\Phi_{22}(\o;\alpha)\geq 0$. Since $u_{11}(\o), u_{22}(\o)<0$ by  the previous item, and $1-H>0$,  
$$\frac{2}{1-H}u_{11}(\o)+\frac{2\alpha}{u(\o)}\leq 0$$
$$ \frac{2}{1-H}u_{22}(\o)+\frac{2\alpha}{u(\o)}\leq 0.$$
Then 
\begin{equation}\label{delta1}
u_{11}(\o)+u_{22}(\o)=\Delta u(\o)\leq -\frac{2\alpha(1-H)}{u(\o)}\cdot
\end{equation}
Finally, equation  (\ref{eq1}) at  $\o$ yields 
\begin{equation}\label{delta2}
\Delta u(\o)=\frac{-2(1-H)}{u(\o)}\cdot
\end{equation} 
Comparing (\ref{delta1}) and (\ref{delta2}), we conclude, $\alpha\leq 1$. Thus if $\alpha\in (1,2]$, we arrive to a contradiction and the theorem is proved.

We analyse the case $\alpha=1$. Denote by $\o_\alpha$  the   the minimum point  of $\Phi(x;\alpha)$. We have proved that $\o_\alpha$  lies in $\partial\Omega$ for all $\alpha\in (1,2]$. By continuity, the point $\o_1$ must be a boundary point, because on the contrary, $\Phi(x;\alpha)$ would be constant for some parameter $\alpha\in (1,2]$. This proves the result for $\alpha=1$ and the proof of Proposition \ref{t-min} is completed. 
\end{enumerate}
\end{proof}

\begin{remark} In case that $u$ is a radial solution, then $u$ can be expressed as  
$$u(r)=-Hm+\sqrt{m^2-r^2},\quad 0\leq r\leq R.$$
It is not difficult to see that  if we denote $u'=u'(r)$, then the functional
$$\Phi(x;\alpha)=\frac{1+u'^2}{(1-H\sqrt{1+u'^2})^2}u^{2\alpha}$$
is constant only when the parameter $\alpha$ is $\alpha=1$. 
\end{remark}

\begin{proof}[Proof of Theorem \ref{t3}]

First suppose that $u$ is not a radial solution. By the proof of Proposition \ref{t-min}, we know that $\Phi(x;1)$ attains its minimum at some point $Q\in\partial\Omega$. Then if $q_M=|Du|(Q)$, we have 
$$\frac{1+|Du|^2)}{(1-H\sqrt{1+|Du|^2})^2}u^2\geq \frac{1+q_M^2}{(1-H\sqrt{1+q_M^2})^2}a^2.$$
We evaluate this inequality at the only critical point $\o$, obtaining
\begin{equation}\label{uo}
\left(\frac{u_M}{a(1-H)}\right)^{2}\geq  \frac{1+q_M^2}{(1-H\sqrt{1+q_M^2})^2}\cdot
\end{equation}
On the other, $\partial\Phi(Q;1)/\partial\n\leq 0$ because $Q$ is the minimum of $\Phi(x;1)$. If $u_n$ and $u_{nn}$ denote  the first and second outward normal derivatives of $u$ along $\partial\Omega$, by the expression of $\Phi_i$ (see \cite[p. 197]{pp}),  we  deduce
\begin{equation}\label{uo2}
\frac{u_nu_{nn}}{(1+u_n^2)(1-H\sqrt{1+u_n^2})}+\frac{ u_n}{u}\leq 0\ \mbox{at $Q$}\cdot
\end{equation}
In normal coordinates, and taking into account that $u$  is constant along $\partial\Omega$, equation  (\ref{eq1}) along $\partial\Omega$ becomes 
$$\frac{u_{nn}}{(1+u_n^2)^{3/2}}+\frac{\kappa u_n}{\sqrt{1+u_n^2}}=\frac{-2}{u}\left(\frac{1}{\sqrt{1+u_n^2}}-H\right).$$
Combining this equation at $Q$ with (\ref{uo2}) and using that $u_n\leq 0$, 
$$\frac{-1}{a}\geq \frac{\kappa(Q)u_n(Q)}{1-H\sqrt{1+u_n^2(Q)}}\cdot$$
Hence, and as $\kappa(Q)\leq\kappa_0$,  
$$\frac{1}{a^2\kappa_0^2}\leq \frac{u_n^2(Q)}{(1-H\sqrt{1+u_n^2(Q)})^2}\cdot$$
As $|Du|^2=u_n^2$ at $Q$, we obtain from (\ref{uo})
$$\left(\frac{u_M}{a(1-H)}\right)^{2}\geq \frac{u_n^2(Q)}{(1-H\sqrt{1+u_n^2(Q)})^2}\geq \frac{1}{a^2\kappa_0^2},$$
proving the result.

Suppose now that $u$ is a radial solution. Then   $u(r)=c_0+\sqrt{m^2-r^2}$, where $m>0$, $c_0=-Hm$ and $0\leq r\leq R$. Since $m>R$,  
$$u_M=u(0)=(1-H) m>(1-H)R=\frac{1-H}{\kappa_0}\cdot$$
This proves the inequality (\ref{um2}) and completes the proof of Theorem \ref{t3}.
\end{proof}

\section*{Acknowledgements} The author has been partially
supported by MEC-FEDER
 grant no. MTM2017-89677-P.
 

 \end{document}